\documentclass[11pt,leqno]{amsart}
\usepackage{amsmath,amsthm,amssymb}
\usepackage{tabularx}

\usepackage{url}

\newcommand{\GF}{{\mathbb F}}

\newcommand{\R}{{\mathbb R}}

\newcommand{\wt}{{\rm wt}}
\newcommand{\supp}{{\rm supp}}

\DeclareMathOperator{\Harm}{Harm}

\usepackage{color} 

\newtheorem{Thm}{Theorem}[section]

\newtheorem{Lem}[Thm]{Lemma}

\newtheorem{Cor}[Thm]{Corollary}

\newtheorem{Prop}[Thm]{Proposition}

\theoremstyle{definition}
\newtheorem{Def}[Thm]{Definition}

\newtheorem{rem}[Thm]{Remark}

\newcommand{\CC}{\mathbb{C}}

\newcommand{\FF}{\mathbb{F}}

\makeatletter
\@addtoreset{equation}{section}

\makeatother

\begin{document}

\title[On on Assmus--Mattson type theorem]
{
On an Assmus--Mattson type theorem 
for Type I and even formally self-dual codes
}

\author{Tsuyoshi Miezaki*}
\thanks{*Corresponding author}
\address{		Faculty of Science and Engineering, 
		Waseda University, 
		Tokyo 169-8555, Japan
}
\email{miezaki@waseda.jp} 
\author{Hiroyuki Nakasora}
\address{Institute for Promotion of Higher Education, Kobe Gakuin University, Kobe
651--2180, Japan}
\email{nakasora@ge.kobegakuin.ac.jp}

\date{}

\maketitle

\begin{abstract}
In the present paper, we give an Assmus--Mattson type theorem
for near-extremal Type I and even formally self-dual codes.
We show the existence of $1$-designs or $2$-designs for these codes. 
As a corollary, 
we prove the uniqueness of a self-orthogonal $2$-$(16,6,8)$ design.
\end{abstract}

{\small
\vspace{-0.15in}

\quad
}


\setcounter{section}{+0}
\section{Introduction}

The Assmus--Mattson theorem derives 
the existence of designs from certain 
conditions 
on the weight distribution of the dual code.
Although the Assmus--Mattson theorem 
cannot be applied to 
a near-extremal Type II code, 
we showed that all the fixed weight supports are 
$1$-designs \cite{Miezaki-Munemasa-Nakasora}. 
We call this an Assmus--Mattson type theorem. 

In the present paper, we state an Assmus--Mattson type theorem
for Type I and even formally self-dual codes. 
Let us explain the first result of the present paper. 
Type I codes are binary self-dual codes that are not doubly even. 
A binary even code with the same weight distribution as its dual code is called even formally self-dual. 
Let $C$ be a Type I or an even formally self-dual code of length $n$.
The minimum distance $d$ of $C$ satisfies the
Mallows-Sloane bound \cite{mallows-sloane}
\[
d \leq 2 \left\lfloor\frac{n}{8}\right\rfloor+2.
\]
We say that $C$ meeting
the Mallows-Sloane bound with equality is extremal. 
If 
\[d=2 \left\lfloor\frac{n}{8}\right\rfloor,\]
$C$ is called near-extremal.
We note that 
the Assmus--Mattson theorem 
cannot be applied to 
near-extremal Type I and even formally self-dual codes. 
However, we show the following. 
\begin{Thm}\label{thm:near}
Let $C$ be a near-extremal Type I code of length $n$ 
and $C'$ be a near-extremal even formally self-dual code of length $n$.
Let $C_{w}$ be the support design of a code $C$ for weight $w$. 
If $n \equiv 0 \pmod{8}$, then $C_{w}$ and $C'_{w}\cup C'^{\perp}_{w}$ are $1$-designs.
\end{Thm}

Let us explain the second result of the present paper. 
We introduce the following notations: 
\begin{align*}
\delta(C)&:=\max\{t\in \mathbb{N}\mid \forall w, 
C_{w} \mbox{ is a } t\mbox{-design}\},\\ 
s(C)&:=\max\{t\in \mathbb{N}\mid \exists w \mbox{ s.t.~} 
C_{w} \mbox{ is a } t\mbox{-design}\}.
\end{align*}
We remark that $\delta(C) \leq s(C)$ holds. 
In our previous papers 
\cite{MN-St-2,{BMN},{extremal design H-M-N}, Miezaki-Munemasa-Nakasora, extremal design2 M-N,support design triply even code 48 M-N, dual support designs,{MN-AAECC}}, 
we considered the possible occurrence of $\delta(C)<s(C)$.
This was motivated by Lehmer's conjecture, which is 
analogous to $\delta(C)<s(C)$ in the theory of lattices 
and vertex operator algebras. 
For the details, 
see \cite{{BM1},{BM2},{BMY},{Lehmer},{Miezaki},{Miezaki2},{Miezaki-Munemasa-Nakasora},
{Venkov},{Venkov2}}. 
In the present paper, we give more examples of $\delta(C)<s(C)$ 
and show the existence of an interesting support $2$-design.




If $n \geq 72$ ($n \geq 40$), there is no near-extremal even formally self-dual (Type I) code of 
length $n \equiv 0 \pmod{8}$ \cite{han-kim}. 
It is not known whether there exists a near-extremal even formally self-dual 
$[48, 24, 12]$ code that is not Type II. 
Furthermore, the existence of near-extremal even formally self-dual $[56, 28, 14]$, $[64, 32, 16]$ codes is open.
It is known that there are exactly 144 inequivalent formally self-dual even $[16,8,4]$ codes,
one of which is Type I and two of which are Type II \cite{B-H}.
Hence, the Type I $[16,8,4]$ code is unique. 
The following theorem 
gives a strengthening of Theorem~\ref{thm:near} for some particular cases. 
\begin{Thm}\label{thm:ex1 2-design}
\begin{enumerate}
\item[{\rm (1)}] 
Let $C$ be a near-extremal Type I code of length $n \equiv 0 \pmod{8}$.
The case of $\delta(C)<s(C)$ occurs 
if and only if $C$ is the unique near-extremal Type I $[16,8,4]$ code. 
Then $C_{6}$ is a $2$-$(16,6,8)$ design, and $C_{10}$ is the complement design of $C_{6}$.

\item[{\rm (2)}] Let $C'$ be a near-extremal even formally self-dual code of length $16$.
Then $C'_{6}\cup C'^{\perp}_{6}$ $($also $C'_{10}\cup C'^{\perp}_{10}$$)$ is a $2$-design.

\item[{\rm (3)}] If there exists a near-extremal even formally self-dual code $C''$ of length $64$, 
then $C''_{28} \cup C''^{\perp}_{28}$ $($also $C''_{36} \cup C''^{\perp}_{36}$$)$ is a $2$-design.
\end{enumerate}

\end{Thm}

We note that there are at least $10^{8}$ non-isomorphic 
$2$-$(16,6,8)$ designs \cite{CRC-MR}.
A $t$-design is called self-orthogonal if the block intersection numbers 
have the same parity as the block size $k$ \cite{Tonchev2}. 
In Theorem \ref{thm:ex1 2-design} (1), $C_{6}$ is a self-orthogonal design.
Let $C$ be the near-extremal Type I $[16,8,4]$ code. 
Then $C_{4}$, $C_{8}$, and $C_{12}$ are not generated by $C$.
We show the following theorem.
\begin{Thm}\label{thm:unique}
There exists the unique self-orthogonal $2$-$(16,6,8)$ design 
$D$ up to isomorphism, and $D$ generates the near-extremal Type I 
$[16,8,4]$ code. 
\end{Thm}

Appendix B in our previous paper \cite{dual support designs} lists the cases where
$\delta(C) < s(C)$ occurs for $4$-weight binary codes.
An example of these cases is given by the unique Type I $[16,8,4]$ code.
\begin{Cor}\label{cor:main1}
Let $C$ be the unique Type I $[16,8,4]$ code. 
Let $C_{0}$ be a doubly-even subcode of $C$. 
Then $C_{0}$ is an example of $n=16, d=4, w=6,10$ 
in \cite[Appendix B. Table of $(d^{\perp},t)=(4,1)$]{dual support designs}.

\end{Cor}


All computer calculations in this paper were done with the help of 
{\sc Magma}~\cite{Magma} and {\sc Mathematica}~\cite{Mathematica}. 


\section{Preliminaries}\label{sec:pre}

\subsection{Background material and terminology}\label{sec:terminology}

Let $\GF_q$ be the finite field of $q$ elements. 
A binary linear code $C$ of length $n$ is a subspace of $\GF_2^n$. 
An inner product $({x},{y})$ on $\FF_2^n$ is given 
by
\[
(x,y)=\sum_{i=1}^nx_iy_i,
\]
where $x,y\in \FF_2^n$ with $x=(x_1,x_2,\ldots, x_n)$ and 
$y=(y_1,y_2,\ldots, y_n)$. 
The duality of a linear code $C$ is defined as follows: 
\[
C^{\perp}=\{{y}\in \FF_{2}^{n}\ | \ ({x},{y}) =0\ \mbox{ for all }{x}\in C\}.
\]
A linear code $C$ is self-dual 
if $C=C^{\perp}$. 
For $x \in\FF_2^n$,
the weight $\wt(x)$ is the number of its nonzero components. 
The minimum distance of code $C$ is 
$\min\{\wt( x)\mid  x \in C,  x \neq  0 \}$. 
A linear code of length $n$, dimension $k$, and 
minimum distance $d$ is called an $[n,k,d]$ code 
(or $[n,k]$ code), and 
the dual code is called an $[n,n-k,d^{\perp}]$ code. 
{

A $t$-$(v,k,{\lambda})$ design (or $t$-design for short) is a pair 
$\mathcal{D}=(X,\mathcal{B})$, where $X$ is a set of points of 
cardinality $v$, and $\mathcal{B}$ is a collection of $k$-element subsets
of $X$ called blocks, with the property that any $t$ points are 
contained in precisely $\lambda$ blocks.

The support of a nonzero vector ${x}:=(x_{1}, \dots, x_{n})$, 
$x_{i} \in \GF_{2} = \{ 0,1 \}$ is 
the set of indices of its nonzero coordinates: ${\rm supp} ({ x}) = \{ i \mid x_{i} \neq 0 \}$\index{$supp (x)$}. 
The support design of a code of length $n$ for a given nonzero weight $w$ is the design 
with points $n$ of coordinate indices and blocks the supports for all codewords of weight $w$.

The following theorem is from Assmus and Mattson \cite{assmus-mattson}; it is one of the 
most important theorems in coding theory and design theory.
\begin{Thm}[\cite{assmus-mattson}] \label{thm:assmus-mattson}
Let $C$ be a binary $[n,k,d]$ linear code and $C^{\bot}$ be the $[n,n-k,d^{\bot}]$ dual code. 
Let $t$ be an integer less than $d$.  
Let $C$ have at most $d^{\bot}-t$ non-zero weights less than or equal to $n-t$. Then 
for each weight $u$ with $d \leq u \leq n-t$, the support design in $C$ is a $t$-design, 
and for each weight $w$ with $d^{\bot} \leq w \leq n$, the support design 
in $C^{\bot}$ is a $t$-design.
\end{Thm}

\subsection{The harmonic weight enumerators}\label{sec:weight enumerators}

A striking generalization of the MacWilliams identity was 
obtained by Bachoc~\cite{Bachoc}, 
who originated the concept of harmonic weight enumerators $W_{C,f}$ 
associated with a harmonic function $f$ of degree $d$, 
and a generalization of the MacWilliams identity. 
For the reader's convenience, we quote from~\cite{Bachoc,Delsarte}
the definitions and properties of discrete harmonic functions 
(for more information the reader is referred to~\cite{Bachoc,Delsarte}).

Let $\Omega=\{1, 2,\ldots,n\}$ be a finite set (which will be the set of coordinates of the code) and 
let $X$ be the set of its subsets, while for all $k= 0,1, \ldots, n$, $X_{k}$ is the set of its $k$-subsets.
We denote by $\R X$, $\R X_k$ the free real vector spaces spanned by, respectively, the elements of $X$, $X_{k}$. 
An element of $\R X_k$ is denoted by
$$f=\sum_{z\in X_k}f(z)z$$
and is identified with the real-valued function on $X_{k}$ given by 
$z \mapsto f(z)$. 

Such an element $f\in \R X_k$ can be extended to an element $\widetilde{f}\in \R X$ by setting, for all $u \in X$,
$$\widetilde{f}(u)=\sum_{z\in X_k, z\subset u}f(z).$$
If an element $g \in \R X$ is equal to some $\widetilde{f}$, for $f \in \R X_{k}$, we say that $g$ has degree $k$. 
The differentiation $\gamma$ is the operator defined by linearity from 
$$\gamma(z) =\sum_{y\in X_{k-1},y\subset z}y$$
for all $z\in X_k$ and for all $k=0,1, \ldots n$, and $\Harm_{k}$ is the kernel of $\gamma$:
$$\Harm_k =\ker(\gamma|_{\R X_k}).$$

\begin{Thm}[{{\cite[Theorem 7]{Delsarte}}}]\label{thm:design}
A set $\mathcal{B} \subset X_{m}$, where $m \leq n$, of blocks is a $t$-design 
if and only if $\sum_{b\in \mathcal{B}}\widetilde{f}(b)=0$ 
for all $f\in \Harm_k$, $1\leq k\leq t$. 
\end{Thm}

In \cite{Bachoc}, the harmonic weight enumerator associated with a binary linear code $C$ was defined as follows. 
\begin{Def}
Let $C$ be a binary code of length $n$ and let $f\in\Harm_{k}$. 
The harmonic weight enumerator associated with $C$ and $f$ is

$$W_{C,f}(x,y)=\sum_{{c}\in C}\widetilde{f}({c})x^{n-\wt({c})}y^{\wt({c})}.$$
\end{Def}

Bachoc proved the following MacWilliams-type equality. 
\begin{Thm}[\cite{Bachoc}] \label{thm: Bachoc iden.} 
Let $W_{C,f}(x,y)$ be 
the harmonic weight enumerator associated with the code $C$ 
and the harmonic function $f$ of degree $k$. Then 
$$W_{C,f}(x,y)= (xy)^{k} Z_{C,f}(x,y)$$ ,
where $Z_{C,f}$ is a homogeneous polynomial of degree $n-2k$ that satisfies
$$Z_{C^{\bot},f}(x,y)= (-1)^{k} \frac{2^{n/2}}{|C|} Z_{C,f} \left( \frac{x+y}{\sqrt{2}}, \frac{x-y}{\sqrt{2}} \right).$$
\end{Thm}

\subsection{Mendelsohn equations}\label{sec:men}

We recall the Mendelsohn equations, 
which were used in \cite{{harada},{harada-kitazume-munemasa}}.
In this paper they will be 
used in the proof of Theorem~\ref{thm:unique}. 
Let $C$ be the code generated by the rows of 
a block-point incidence matrix $A$ of a
$t$-$(v,k,\lambda)$ design $D=(X,\mathcal{B})$. 
Let $u \in C^{\perp}$ be a vector of weight $m>0$. 
Denote by $n_{i}$ the number of blocks of 
$D$ that meet $\supp(u)$ in exactly $i$ points. 
Then we have the following system of equations.
\begin{Thm}[\cite{{{Mendelsohn}},{Tonchev2}}]\label{thm:men}
\[
\sum_{i=0}^{\min\{k,m\}}\binom{i}{j}n_i
=\lambda_j\binom{m}{j}\ (j=0,1,\ldots,t). 
\]
\end{Thm}

\section{Proof of Theorem~\ref{thm:near}}

By Bachoc \cite[Corollary 2.2, Lemma 3.2]{Bachoc} we have the following Lemma 
for an even formally self-dual code.

\begin{Lem}\label{lem:Zc}
Let ${C}$ be an even formally self-dual code of length $n$ 
and $f\in {\rm Harm}_{t}$.

\begin{align*}
& Z_{{C},f}+ Z_{{C}^{\perp},f} \\
&=
\begin{cases}
\displaystyle Q_{8}\sum_{i=0}^{n/8} (x^{2}+y^{2})^{n/2-(t+4)-4i}(x^{2}y^{2}(x^{2}-y^{2})^{2})^{i},
\text{if $t$ is odd}\\
\displaystyle \sum_{i=0}^{n/8} (x^{2}+y^{2})^{n/2-t-4i}(x^{2}y^{2}(x^{2}-y^{2})^{2})^{i}, 
\text{if $t$ is even,}
\end{cases}
\end{align*}
where $Q_{8}=xy(x^{6}-7x^{4}y^{2}+7x^{2}y^{4}-y^{6})$.

\end{Lem}
\begin{proof}
By Bachoc \cite[Corollary 2.2]{Bachoc}, 
$Z_{{C},f}+ Z_{{C}^{\perp},f}$ is a relative invariant for the group 
\[
G:=
\left\langle 
T_1:=
\frac{1}{\sqrt{2}}
\begin{bmatrix}
1&1\\
1&-1
\end{bmatrix}, 
T_3:=
\begin{bmatrix}
1&0\\
0&-1
\end{bmatrix}
\right\rangle 
\]
with respect to the character $\chi$: 
\[
\chi(T_1)=(-1)^t,\ 
\chi(T_3)=(-1)^t. 
\]
By Bachoc \cite[Lemma 3.2]{Bachoc}, 
the space of relative invariants with respect to $G$ and $\chi$ 
is given as follows: 
\[
\begin{cases}
\CC[x^2+y^2, x^2y^2(x^2-y^2)^2]\ \mbox{if $t$ is even},\\
Q_8\CC[x^2+y^2, x^2y^2(x^2-y^2)^2]\ \mbox{if $t$ is odd}. 
\end{cases}
\]
The proof is complete. 
\end{proof}

\begin{proof}[Proof of Theorem \ref{thm:near}]

Let $n=8m+r$ ($r=0,2,4,6$), the minimum weight of $C$ or $C'$ be $d=2m$ 
and $f\in\Harm_{t}$. 
Since $C$ is self-dual, $Z_{C,f}=Z_{C^{\perp},f}$.
We have \[W_{C' \cup C'^{\perp}}=(xy)^{t}(Z_{C',f}+ Z_{C'^{\perp},f}).\]

By Lemma~\ref{lem:Zc} we have
\begin{align*}
W=& 2W_{C}= W_{C' \cup C'^{\perp}}\\
&=
\begin{cases}
\displaystyle (xy)^{t}Q_{8}\sum_{i=0}^{n/8} (x^{2}+y^{2})^{n/2-(t+4)-4i}(x^{2}y^{2}(x^{2}-y^{2})^{2})^{i},
\text{if $t$ is odd}\\
\displaystyle (xy)^{t}\sum_{i=0}^{n/8} (x^{2}+y^{2})^{n/2-t-4i}(x^{2}y^{2}(x^{2}-y^{2})^{2})^{i}, 
\text{if $t$ is even.}
\end{cases}
\end{align*}
For $t=1$, the degree of $(x^{2}+y^{2})$ in $W$ is 
\[4m+\frac{r}{2} -5-4(m-1)=\frac{r}{2}-1.\] 
Hence $W=0$ for $r=0$. 
\end{proof}

\section{The case  $\delta(C) < s(C)$ occurs}

\begin{Lem}\label{lem:poly. zero}

Let $R=(x^{4}+2x^{2}y^{2}+y^{4})(x^{2}-y^{2})^{\alpha}$ with $\alpha<16$.
If the coefficients of 
$x^{2\alpha+4-2i}y^{2i}$ in $R$ are equal to $0$ 
for $0 \leq i \leq (\alpha+2)/2$, 
then $(\alpha,i)=(2,1), (7,2)$, or $(14,6)$.
\end{Lem}
\begin{proof}
We checked numerically using {\sc Mathematica} \cite{Mathematica}. 
\end{proof}




\begin{proof}[Proof of Theorem \ref{thm:ex1 2-design}]

Let $C$ be a near-extremal Type I code of length $n=8m$
and $C'$ be a near-extremal even formally self-dual code of length $n=8m$.
Let $f\in \mbox{Harm}_{2}$.
By Lemma~\ref{lem:Zc} we have
\begin{align*}
W&=W_{C \cup C^{\perp}}=2W_{C'} \\
&=(xy)^{2}\sum_{i=0}^{\frac{n}{8}} (x^{2}+y^{2})^{4m-2-4i}(x^{2}y^{2}(x^{2}-y^{2})^{2})^{i}\\ 
&=(xy)^{2m}\sum_{i=0}^{\frac{n}{8}} (x^{2}+y^{2})^{2}(x^{2}-y^{2})^{2m-2}\\ 
&=(xy)^{2m}\sum_{i=0}^{\frac{n}{8}} (x^{4}+2x^{2}y^{2}+y^{4})(x^{2}-y^{2})^{2m-2}.
\end{align*}

By Lemma~\ref{lem:poly. zero},
if the coefficients of 
$x^{6m-2i}y^{2m+2i}$ in $W$ are equal to $0$ for $0 \leq i \leq m$, 
then $(m,i)=(2,1), (8,6)$.

In the case $(m,i)=(2,1)$, the length of $C$ and $C'$ is $16$.
Let $C$ and $C'$ be $[16,8,4]$ codes.
Then $C_{6}$ and $C'_{6} \cup C'^{\perp}_{6}$ are $2$-designs. 

In the case $(m,i)=(8,6)$, the length of $C$ and $C'$ is $64$.
There is no near-extremal Type I code of length $64$.
Let $C'$ be a $[64,32,16]$ code.
Then $C'_{28} \cup C'^{\perp}_{28}$ is a $2$-design. 
\end{proof}

\begin{rem}
Using {\sc Magma} \cite{Magma} and \cite{codes}, we checked that Theorem 1.2 (1) holds. 
\end{rem}

\section{Proof of Theorem~\ref{thm:unique}}


We quote the following lemma. 
\begin{Lem}[{{\cite[Page 3, Proposition 1.4]{CL}}}]\label{lem: divisible}
Let $\lambda(S)$ be the number of blocks containing a given set $S$
of $s$ points in a combinatorial $t$-$(v,k,\lambda)$ design, where $0\leq s\leq t$. Then
\[
\lambda(S)\binom{k-s}{t-s}
=
\lambda\binom{v-s}{t-s}. 
\]
\end{Lem}

By Lemma \ref{lem: divisible}, for a $t$-$(v,k,{\lambda})$ design, 
every $i$-subset of points $(i \leq t)$ is contained in exactly
\[
\lambda_{i}=\lambda  {\displaystyle\binom{v-i}{t-i}}\displaystyle\Big{/}{\displaystyle\binom{k-i}{t-i}}
\]
blocks. 
For a $2$-$(16,6,8)$ design, we have 
\[\lambda_{0}=64, \lambda_{1}=24, \lambda_{2}=8.\]
\begin{Prop}\label{prop:block int.}
Let $D=(X,\mathcal{B})$ be a self-orthogonal $2$-$(16,6,8)$ design.
For a block $B$ of $\mathcal{B}$, let $m_{i}$ 
be the number of other blocks that meet $B$ in exactly $i$ points. 
Then \[m_{0}=3, m_{2}=51, m_{4}=9.\]
\end{Prop}
\begin{proof}
By Theorem \ref{thm:men} we have the system of equations:
\begin{eqnarray*}
\sum_{i=0}^{3} \binom{2i}{j} m_{2i}=\lambda_{j} \binom{6}{j} ~~ (j=0,1,2). \label{eqn:counting for block intersection}
\end{eqnarray*}
Then we have a unique solution.
\end{proof}

Let $C$ be the code generated by the rows of 
a block-point incidence matrix $A$ of a
self-orthogonal $2$-$(16,6,8)$ design $D=(X,\mathcal{B})$. 
Then $C$ is an even self-orthogonal code of length $16$. 

\begin{Lem}\label{lem:min.wt}
The minimum weight of $C^{\perp}$ is $4$.
\end{Lem}
\begin{proof}
Since $D$ has $\lambda_{1}-\lambda_{2}=16$, for any two points 
\[a \in \binom{X}{2}, \]
there are $16 \times 2$ blocks $B \in \mathcal{B}$ such that $|a \cap B|=1$.
Hence $C^{\perp}$ does not have a codeword of weight $2$. 
By Proposition~\ref{prop:block int.} there are some two blocks 
$B, B' \in \mathcal{B}$ such that $|B \cap B'|=4$.
Let $x,y\in C$ such that ${\rm supp}(x)=B$ and $\supp(y)=B'$. 
Since we have ${\wt}(x+y)=4$ and $x+y\in C\subset C^\perp$, 
the minimum weight of $C^{\perp}$ is $4$. 
\end{proof}

\begin{Lem}\label{lem:self-dual}
The code $C$ is self-dual.
\end{Lem}
\begin{proof}
From $|C_6|\geq 64$ we have
\[
1+|C_6|+|C_{10}|+1 > 2^7. 
\]
Hence $\dim C \geq 8$. 
By the fact that $C\subset C^\perp$ and $\dim C^\perp\leq 8$ 
we have $C=C^\perp$. 
\end{proof}
}

\begin{proof}[Proof of Theorem \ref{thm:unique}]
Let $C$ be the code generated by the rows of a block-point incidence matrix $A$ of a
self-orthogonal $2$-$(16,6,8)$ design. 
By Lemmas~\ref{lem:min.wt} and \ref{lem:self-dual} $C$ is a Type I $[16,8,4]$ code. 
Since there exists a unique Type I $[16,8,4]$ code, 
the self-orthogonal $2$-$(16,6,8)$ design is unique up to isomorphism.
\end{proof}


\section*{Acknowledgments}

The authors would like to thank the anonymous reviewers 
for their beneficial comments on an earlier version of the manuscript. 
The first named author was supported by JSPS KAKENHI (22K03277). 



\end{document}